\documentclass[11pt,reqno]{amsart}
\usepackage[T1]{fontenc}
\usepackage[utf8]{inputenc}
\usepackage{lmodern}
\usepackage[a4paper,left=27mm,right=27mm,top=27mm,bottom=28mm]{geometry}
\usepackage{amsmath,amssymb,amsthm,mathrsfs}
\usepackage{microtype}
\usepackage{enumitem}
\usepackage{needspace}
\usepackage{hyperref}
\hypersetup{colorlinks=true,linkcolor=black,citecolor=black,urlcolor=black,
  pdftitle={Direct proofs for spaces with cyclic parallel Ricci tensor},
  pdfsubject={Gradient Ricci solitons, quasi-Einstein manifolds, and critical spaces},
  pdfauthor={},bookmarksnumbered=true}
\newcommand{\Ric}{\operatorname{Ric}}
\newcommand{\Rm}{\operatorname{Rm}}
\newcommand{\tr}{\operatorname{tr}}
\newcommand{\diver}{\operatorname{div}}
\makeatletter
\newcommand{\owedge}{\mathbin{\mathpalette\owedge@{}}}
\newcommand{\owedge@}[2]{\vcenter{\hbox{\ooalign{%
  \hfil$#1\wedge$\hfil\cr$#1\bigcirc$\cr}}}}
\makeatother
\newtheorem{theorem}{Theorem}
\newtheorem{proposition}{Proposition}
\newtheorem{lemma}{Lemma}

\theoremstyle{remark}
\newtheorem{remark}{Remark}
\numberwithin{equation}{section}
\allowdisplaybreaks[2]
\setlist[enumerate]{leftmargin=*,itemsep=3pt,topsep=5pt}
\title[CYCLIC PARALLEL RICCI TENSOR]{Direct proofs for spaces with\texorpdfstring{\\}{} cyclic parallel Ricci tensor}
\author{
 Fengyuan Zhang \\
  Academy of mathematics and Systems Science\\
  \texttt{zhangfengyuan@amss.ac.cn} \\
  }

\begin{document}
\begin{abstract}
We give direct proofs of rigidity results for  gradient Ricci
solitons, compact quasi-Einstein manifolds with boundary, and critical
spaces with cyclic parallel Ricci tensor. For solitons and
quasi-Einstein manifolds, successive traces of first-order pointwise
identities yield quadratic equations for the Ricci tensor and its
modification. These equations imply radial flatness and radial
$Q$-flatness, respectively, and the corresponding global rigidity
theorems apply. For critical spaces, comparison of two expressions for
the Laplacian of the Hessian
gives a pointwise tensor identity whose scalar contraction yields a
weighted divergence formula. Integration then gives the vacuum static,
CPE, and Miao--Tam conclusions. The same comparison also proves Ricci
parallelism pointwise for solitons and quasi-Einstein manifolds.
\end{abstract}
\maketitle

\section{Introduction}

The cyclic parallel Ricci condition is expressed in components by
\begin{equation}\label{eq:cyclic}
D_{ijk}=\nabla_iR_{jk}+\nabla_jR_{ki}+\nabla_kR_{ij}=0.
\end{equation}
The author~\cite{Zhang} studied gradient shrinking Ricci solitons, compact
quasi-Einstein manifolds with boundary, and critical spaces satisfying
\eqref{eq:cyclic}, using integral identities to establish parallel Ricci
curvature. We give a different approach based on pointwise tensor
identities. For solitons and quasi-Einstein manifolds, the reduction to
radial flatness is entirely pointwise. For critical spaces, a pointwise
second-order tensor equation is derived first, and its scalar contraction
is then integrated. The soliton classification includes the shrinking,
steady, and expanding cases; the steady case reduces directly to
Ricci-flatness.

The geometric classifications are completed using the rigidity theorems
of Petersen--Wylie~\cite{PW} and He--Petersen--Wylie~\cite{HPW14},
the Hessian characterization of a round sphere due to Obata~\cite{Obata},
and the Einstein Miao--Tam classification~\cite{MT11}. The local
calculations and the hypotheses needed for these applications are
included in the proofs. 

\subsection{Three types of space}

A gradient Ricci soliton is a complete Riemannian manifold $(M^n,g)$
with a smooth potential $f$ satisfying
\begin{equation}\label{eq:soliton}
R_{ij}+\nabla_i\nabla_j f=\lambda g_{ij}.
\end{equation}
It is shrinking, steady, or expanding according as $\lambda>0$,
$\lambda=0$, or $\lambda<0$. The non-steady case means $\lambda\ne0$.

For $m>1$, a quasi-Einstein manifold with boundary has a smooth
potential $u$ such that
\begin{equation}\label{eq:quasi}
\left\{
\begin{aligned}
\nabla_i\nabla_j u&=\frac{u}{m}(R_{ij}-\lambda g_{ij})
   &&\text{in }M,\\
u&>0 &&\text{in }\operatorname{int}(M),\\
u&=0 &&\text{on }\partial M.
\end{aligned}
\right.
\end{equation}
Here $m$ and $\lambda$ are constants. We use the potential $u$ for
quasi-Einstein manifolds and $f$ for solitons and critical spaces.

The critical spaces satisfy the unified equation
\begin{equation}\label{eq:critical}
\nabla_i\nabla_j f+\frac{R}{n(n-1)}fg_{ij}
  =(f+a)\mathring R_{ij}+bg_{ij},
\qquad
\mathring R_{ij}=R_{ij}-\frac Rn g_{ij},
\end{equation}
where $a,b\in\mathbb R$ are constants. The three cases considered here
are
\begin{equation}\label{eq:critical-cases}
\begin{aligned}
(a,b)&=(0,0) &&\text{for vacuum static spaces},\\
(a,b)&=(1,0) &&\text{for CPE metrics},\\
(a,b)&=\left(0,-\frac{1}{n-1}\right)
  &&\text{for Miao--Tam critical metrics}.
\end{aligned}
\end{equation}
For the closed static and CPE results below, the potential is required
to be nonconstant. A closed manifold is compact and has no boundary.
For a Miao--Tam critical metric, $f$ is positive in the interior and
vanishes on the boundary.

\subsection{Pointwise identities and comparison with earlier proofs}
In ~\cite{Zhang}, the author focused on constructing unconditional integral identities on various spaces and then used the cyclic Ricci parallelism condition to control the curvature terms arising in the integrals. The difference from that integral approach is that here the cyclic condition first yields closed tensor identities along the gradient of the potential. For a gradient Ricci soliton, we obtain
\begin{equation}\label{eq:intro-soliton-identity}
\nabla_kf\,\nabla_kR_{ij}
=2\lambda R_{ij}-2R_{ik}R_{kj}.
\end{equation}
Since $R$ is constant, successive traces give
\[
|\Ric|^2=\lambda R,\qquad
\tr(\Ric^3)=\lambda^2R,\qquad
\tr(\Ric^4)=\lambda^3R.
\]
Consequently,
\[
\left|R_{ik}R_{kj}-\lambda R_{ij}\right|^2=0.
\]
This proves radial flatness without integration or a cutoff argument.
Neither the tensor identity nor the trace recursion uses the sign of
$\lambda$. Petersen--Wylie's rigidity theorem~\cite[Theorem~1.2]{PW}
therefore gives both the shrinking and expanding classifications in
Theorem~\ref{thm:soliton}. When $\lambda=0$, the first trace already gives
$\Ric=0$, and the steady classification follows from
\cite[Proposition~3.2]{PW}.

For quasi-Einstein manifolds, let $\rho$ and $P$ be as in
\eqref{eq:rho-P}. The corresponding first-order identity is
\begin{equation}\label{eq:intro-quasi-identity}
\nabla_ku\,\nabla_kP_{ij}
=\frac{2u}{m}\bigl(P_{ik}P_{kj}-(\lambda-\rho)P_{ij}\bigr).
\end{equation}
The same trace argument yields
$P_{ik}P_{kj}=(\lambda-\rho)P_{ij}$ and then radial $Q$-flatness.
The global classification follows from
He--Petersen--Wylie~\cite[Theorem~1.11]{HPW14}.

For the critical equation \eqref{eq:critical}, comparison of the two
expressions for $\Delta(\nabla_i\nabla_jf)$ gives
\begin{equation}\label{eq:intro-critical-identity}
2(f+a)\Delta R_{ij}+4\nabla_kf\,\nabla_kR_{ij}
+\left(\frac{aR}{n-1}+nb\right)\mathring R_{ij}=0.
\end{equation}
Its contraction with $(f+a)\mathring R_{ij}$ gives the common divergence
formula for the three critical-space results. In this case the final
rigidity argument still uses compactness and integration; it is not a
purely pointwise classification.

\Needspace{29\baselineskip}
\subsection{Main results}

All manifolds in the statements below are smooth and connected, and
$n\geq3$. Simple connectedness is stated explicitly where a direct
product conclusion uses it. In all five main theorems, a solution is
called nontrivial when its given potential $f$ or $u$ is nonconstant.
A zero-dimensional factor $N^0$ is understood to be a point.

\begin{theorem}[Gradient Ricci solitons]
\label{thm:soliton}
Let $(M^n,g,f)$ be a complete nontrivial gradient Ricci soliton satisfying
\eqref{eq:soliton} with cyclic parallel Ricci tensor. Then the Ricci
tensor is parallel. If $\lambda\ne0$, then
\[
(M^n,g)\cong\bigl(N^{n-k}\times\mathbb R^k\bigr)/\Gamma,
\qquad 1\leq k\leq n,
\qquad \Ric_N=\lambda g_N,
\]
where $N$ is complete and simply connected, and $\Gamma$ acts freely
and properly discontinuously by product isometries preserving the
lifted potential. After translating the Euclidean factor and adding a
constant, the lifted potential is $f=\frac{\lambda}{2}|x|^2$, and the
action on $\mathbb R^k$ is orthogonal. The quotient is trivial when $M$
is simply connected. More precisely:
\begin{enumerate}[label=\textup{(\roman*)}]
\item If $\lambda>0$, the soliton is shrinking, $N$ is compact, and
$\Gamma$ is finite.
\item If $\lambda<0$, the soliton is expanding, $N$ has negative
Einstein constant when it has positive dimension, and $\Gamma$ need
not be finite.
\item If $\lambda=0$, the soliton is steady and $\Ric=0$. There is a
global isometric splitting of the original manifold
\[
(M^n,g)\cong(N^{n-1}\times\mathbb R,g_N+dt^2),
\qquad \Ric_N=0,
\]
where $N$ is complete and connected. In this splitting,
$f(q,t)=ct+b$ for constants $c\ne0$ and $b\in\mathbb R$.
No simple connectedness assumption is required in this case.
\end{enumerate}
\end{theorem}

\begin{theorem}\label{thm:quasi}
Let $(M^n,g,u)$ be a nontrivial compact $m$-quasi-Einstein manifold
with $m>1$, nonempty smooth boundary, and cyclic parallel Ricci
tensor. Then $\lambda>0$, the Ricci tensor is parallel, and the
universal cover is isometric to
\[
\mathbb S^k_+\!\left(\sqrt{\frac{m+k-1}{\lambda}}\right)
   \times N^{n-k},
\qquad 1\leq k\leq n,
\qquad \Ric_N=\lambda g_N.
\]
The factor $N$ is closed when it has positive dimension, and $M$
is a finite quotient preserving the lifted potential. The case
$n-k=1$ is impossible. If $k=n$, $M$ is the round hemisphere itself.
For $k=1$, the hemisphere factor is an interval of length
$\pi\sqrt{m/\lambda}$.
\end{theorem}

\begin{theorem}\label{thm:static}
Let $(M^n,g,f)$ be a nontrivial closed simply connected vacuum static
space with cyclic parallel Ricci tensor. Then
$R>0$, the Ricci tensor is parallel, and
\[
(M^n,g)=
\mathbb S^k\!\left(\sqrt{\frac{k(n-1)}R}\right)
  \times N^{n-k},
\qquad 2\leq k\leq n,
\qquad \Ric_N=\frac{R}{n-1}g_N,
\]
where $N$ is a closed Einstein manifold when it has positive
dimension. The potential depends only on the sphere factor and is
a nonzero multiple of a height function.
\end{theorem}

The parallel-Ricci conclusion for vacuum static spaces does not
require simple connectedness. The corresponding universal-cover
description, including the possible one-dimensional active factor,
is given after its proof.

\begin{theorem}\label{thm:cpe}
Let $(M^n,g,f)$ be a nontrivial closed CPE metric with cyclic parallel
Ricci tensor and scalar curvature $R=n(n-1)$.
Then $(M^n,g)$ is isometric to the unit round sphere $\mathbb S^n(1)$.
\end{theorem}

\begin{theorem}\label{thm:mt}
Let $(M^n,g,f)$ be a nontrivial compact oriented Miao--Tam critical
metric with nonempty smooth boundary and cyclic parallel Ricci tensor.
Assume $f>0$ in the interior and $f=0$ on the boundary. Then the
metric is Einstein, and $(M^n,g)$ is isometric to a geodesic ball
in a simply connected space form of sectional curvature
$R/[n(n-1)]$. According to the sign of $R$, the ambient space form
is Euclidean, hyperbolic, or spherical.
\end{theorem}

\Needspace{17\baselineskip}
\section{Preliminaries}

\subsection{Notation}

We write $\Rm=\{R_{ijkl}\}$, $\Ric=\{R_{ij}\}$, and
$\mathring R_{ij}=R_{ij}-\frac Rn g_{ij}$. Repeated indices are
contracted with $g$, and $\Delta$ denotes the rough Laplacian.
Our curvature sign convention is fixed by
\[
\nabla_i\nabla_j\nabla_k f-\nabla_j\nabla_i\nabla_k f
=R_{ijkl}\nabla_lf.
\]
All integrals use the Riemannian volume element $dV$.
Powers and traces of symmetric tensors are taken after raising one
index. In particular,
\begin{equation}\label{eq:scalar-powers}
\begin{aligned}
\tr(\Ric^2)&=|\Ric|^2,
&\tr(P^2)&=|P|^2,\\
\tr(\Ric^3)&=R_{ij}R_{jk}R_{ki},
&\tr(P^3)&=P_{ij}P_{jk}P_{ki},\\
\tr(\Ric^4)&=R_{ij}R_{jk}R_{kl}R_{li},
&\tr(P^4)&=P_{ij}P_{jk}P_{kl}P_{li}.
\end{aligned}
\end{equation}
Tensor products with two free indices remain in components, such as
$R_{ik}R_{kj}$ and $P_{ik}P_{kj}$.

\subsection{Consequences of the cyclic condition}

\begin{lemma}\label{lemma:cyclic}
If $D_{ijk}=0$, then the scalar curvature is constant and
\begin{equation}\label{eq:cyclic-divergence}
\nabla_iR=0,\qquad \nabla_jR_{ij}=0.
\end{equation}
Moreover,
\begin{equation}\label{eq:cyclic-laplacian}
\Delta R_{ij}=2R_{ikjl}R_{kl}-2R_{ik}R_{kj}.
\end{equation}
With the Cotton tensor written as
$C_{ijk}=\nabla_iR_{jk}-\nabla_jR_{ik}$ under constant scalar
curvature, one also has
\begin{equation}\label{eq:cyclic-cotton}
|C|^2=3|\nabla\Ric|^2,
\qquad
C_{ijk}R_{jk}=\frac34\nabla_i|\Ric|^2.
\end{equation}
\end{lemma}

\begin{proof}
Tracing \eqref{eq:cyclic} in its last two indices gives
\[
0=g^{jk}D_{ijk}
  =\nabla_iR+2\nabla_jR_{ij}=2\nabla_iR.
\]
This proves \eqref{eq:cyclic-divergence}.
Taking the divergence of the full cyclic condition, we find
\[
0=\nabla_k\nabla_iR_{jk}
  +\nabla_k\nabla_jR_{ki}+\Delta R_{ij}.
\]
The first term can be computed without a derivative of the scalar
curvature:
\begin{align*}
\nabla_k\nabla_iR_{jk}
&=\nabla_i\nabla_kR_{jk}
  +R_{kijl}R_{lk}+R_{kikl}R_{jl}\\
&=-R_{ikjl}R_{kl}+R_{ik}R_{kj}.
\end{align*}
The second term gives the same expression. Substitution proves
\eqref{eq:cyclic-laplacian}.

Expanding the squared norm of the cyclic condition and relabelling
the summed indices yields
\[
0=\frac{1}{3}|D|^2=|\nabla\Ric|^2
       +2\nabla_iR_{jk}\nabla_jR_{ik}.
\]
It follows that
\[
|C|^2
=2|\nabla\Ric|^2
 -2\nabla_iR_{jk}\nabla_jR_{ik}
=3|\nabla\Ric|^2.
\]
Finally, contracting $D_{ijk}=0$ with $R_{jk}$ gives
\[
0=\frac12\nabla_i|\Ric|^2+2\nabla_jR_{ik}R_{jk}.
\]
Therefore
\[
C_{ijk}R_{jk}
=\frac12\nabla_i|\Ric|^2-\nabla_jR_{ik}R_{jk}
=\frac34\nabla_i|\Ric|^2.
\]
\end{proof}

\subsection{The Hessian--Laplacian formula}

We use the standard commutation formula
\cite[Lemma~3 ]{BDRR17}:
\begin{lemma}\label{lemma:hessian-laplacian}
For any smooth function $f$,
\begin{equation}\label{eq:hessian-laplacian}
\begin{aligned}
\Delta(\nabla_i\nabla_jf)
={}&\nabla_i\nabla_j\Delta f
 +R_{ik}\nabla_k\nabla_jf+R_{jk}\nabla_i\nabla_kf\\
&-2R_{ikjl}\nabla_k\nabla_lf\\
&+(\nabla_iR_{jk}+\nabla_jR_{ik}-\nabla_kR_{ij})\nabla_kf.
\end{aligned}
\end{equation}
\end{lemma}

\begin{proof}
First commute the inner derivatives, and then commute derivatives
of the Hessian as a two-tensor. Using Ricci identity, we get
\begin{align*}
\Delta(\nabla_i\nabla_jf)
={}&\nabla_k\bigl(\nabla_i\nabla_k\nabla_jf
                      +R_{kijl}\nabla_lf\bigr)\\
={}&\nabla_i\nabla_k\nabla_k\nabla_jf
 +R_{il}\nabla_l\nabla_jf
 +2R_{kijl}\nabla_k\nabla_lf
 +(\nabla_kR_{kijl})\nabla_lf.
\end{align*}
The remaining divergence formulas are
\[
\nabla_k\nabla_k\nabla_jf=\nabla_j\Delta f+R_{jl}\nabla_lf,
\qquad
\nabla_kR_{kijl}=\nabla_jR_{il}-\nabla_lR_{ij}.
\]
Substitution, followed by $R_{kijl}=-R_{ikjl}$ and a relabelling
of the contracted indices, gives \eqref{eq:hessian-laplacian}.
\end{proof}

\

\subsection{The tensors for quasi-Einstein manifolds}

Following He--Petersen--Wylie~\cite{HPW12}, we set
\begin{equation}\label{eq:rho-P}
\rho=\frac{(n-1)\lambda-R}{m-1},
\qquad P=\Ric-\rho g.
\end{equation}
Consequently,
\begin{equation}\label{eq:trace-P}
\tr(P)=(n-1)\lambda-(m+n-1)\rho.
\end{equation}
The modified curvature tensor is
\begin{equation}\label{eq:Q-definition}
Q=\Rm+\frac1mP\owedge g
        +\frac{\rho-\lambda}{2m}g\owedge g,
\end{equation}
where, in components,
\[
(P\owedge g)_{ijkl}
=P_{ik}g_{jl}+P_{jl}g_{ik}-P_{il}g_{jk}-P_{jk}g_{il}.
\]
The radial identity of He--Petersen--Wylie~\cite[Proposition~6.2]{HPW12}
takes the form
\begin{equation}\label{eq:Q-radial}
\begin{aligned}
Q_{ijkl}\nabla_lu
={}&\frac{u}{m}(\nabla_iP_{jk}-\nabla_jP_{ik})\\
&-\frac1m(P_{jl}g_{ik}-P_{il}g_{jk})\nabla_lu.
\end{aligned}
\end{equation}
Its constant-scalar-curvature specialization is derived directly in
the proof of Theorem~\ref{thm:quasi}.

\section{Proofs of the main results}

\subsection{Gradient Ricci solitons}

\begin{proof}[Proof of Theorem~\ref{thm:soliton}]
By Lemma~\ref{lemma:cyclic}, the scalar curvature $R$ is constant.
It is a basic fact that every gradient Ricci soliton satisfies
\[
\nabla_iR=2R_{ij}\nabla_jf.
\]
Hence
\begin{equation}\label{eq:soliton-radial-Ric}
R_{ij}\nabla_jf=0.
\end{equation}

Differentiating \eqref{eq:soliton-radial-Ric}, with $j$ as the
free index, yields
\[
0=\nabla_i(R_{jk}\nabla_kf)
  =(\nabla_iR_{jk})\nabla_kf+R_{jk}\nabla_i\nabla_kf.
\]
Substitution of $\nabla_i\nabla_kf=\lambda g_{ik}-R_{ik}$ gives
\begin{equation}\label{eq:soliton-differentiated-annihilation}
(\nabla_iR_{jk})\nabla_kf
=R_{ik}R_{kj}-\lambda R_{ij}.
\end{equation}
The same identity with $i$ and $j$ interchanged has the same
right-hand side. We now contract the full cyclic condition with
$\nabla_kf$:
\[
0=D_{ijk}\nabla_kf
 =(\nabla_iR_{jk})\nabla_kf
  +(\nabla_jR_{ki})\nabla_kf
  +(\nabla_kR_{ij})\nabla_kf.
\]
Using \eqref{eq:soliton-differentiated-annihilation} in the first
two terms, we obtain the pointwise identity
\begin{equation}\label{eq:soliton-flow-Ric}
\nabla_kf\,\nabla_kR_{ij}
=2\lambda R_{ij}-2R_{ik}R_{kj}.
\end{equation}
This step uses the full cyclic condition; constancy of the scalar
curvature alone does not give \eqref{eq:soliton-flow-Ric}.

Taking the trace of \eqref{eq:soliton-flow-Ric}, we have
\[
0=\langle\nabla f,\nabla R\rangle
 =2\lambda R-2|\Ric|^2,
\]
and hence
\begin{equation}\label{eq:soliton-second-trace}
|\Ric|^2=\lambda R.
\end{equation}
If $\lambda=0$, this gives $\Ric=0$ and $\nabla^2f=0$. Since $f$ is
nonconstant, the global splitting and affine potential in part
\textup{(iii)} follow from \cite[Proposition~3.2]{PW}. Henceforth,
assume $\lambda\ne0$.

Since $\lambda$ and $R$ are constant, $|\Ric|^2$ is constant.
Contracting \eqref{eq:soliton-flow-Ric} with $2R_{ij}$ therefore
gives
\[
\begin{aligned}
0
&=\langle\nabla f,\nabla|\Ric|^2\rangle\\
&=2R_{ij}\nabla_kf\,\nabla_kR_{ij}\\
&=4\lambda|\Ric|^2-4\tr(\Ric^3).
\end{aligned}
\]
Thus
\begin{equation}\label{eq:soliton-third-trace}
\tr(\Ric^3)=\lambda^2R.
\end{equation}
The scalar function $\tr(\Ric^3)$ is now constant as well. The product
rule and symmetry of the Ricci tensor yield
\[
\begin{aligned}
0
&=\langle\nabla f,\nabla\tr(\Ric^3)\rangle\\
&=3R_{ik}R_{kj}\,\nabla_\ell f\,\nabla_\ell R_{ij}\\
&=6\lambda\tr(\Ric^3)-6\tr(\Ric^4).
\end{aligned}
\]
It follows that
\begin{equation}\label{eq:soliton-fourth-trace}
\tr(\Ric^4)=\lambda^3R.
\end{equation}
Combining \eqref{eq:soliton-second-trace}--\eqref{eq:soliton-fourth-trace},
we obtain
\[
\begin{aligned}
\left|R_{ik}R_{kj}-\lambda R_{ij}\right|^2
&=\tr(\Ric^4)-2\lambda\tr(\Ric^3)+\lambda^2|\Ric|^2\\
&=\lambda^3R-2\lambda^3R+\lambda^3R\\
&=0.
\end{aligned}
\]
Consequently,
\begin{equation}\label{eq:soliton-Ric-polynomial}
R_{ik}R_{kj}=\lambda R_{ij}.
\end{equation}
Every eigenvalue of $\Ric$ is therefore either $0$ or $\lambda$.
Thus
\[
0\leq\Ric\leq\lambda g\quad\text{if }\lambda>0,
\qquad
\lambda g\leq\Ric\leq0\quad\text{if }\lambda<0.
\]
The calculation above is valid for both signs of $\lambda$.
Substituting
\eqref{eq:soliton-Ric-polynomial} into
\eqref{eq:soliton-flow-Ric} also gives
\begin{equation}\label{eq:soliton-flow-zero}
\nabla_kf\,\nabla_kR_{ij}=0.
\end{equation}
No division by $|\nabla f|$ has been used, so all these identities
hold at the critical points of $f$ as well.

We next verify the radial flatness needed for the global
classification. With our curvature convention, differentiation
of \eqref{eq:soliton} and the Ricci commutation identity give
\begin{equation}\label{eq:soliton-curvature-gradient}
R_{ijkl}\nabla_lf=\nabla_jR_{ik}-\nabla_iR_{jk}.
\end{equation}
Contracting with $\nabla_jf$ and differentiating
\eqref{eq:soliton-radial-Ric}, we find
\[
\begin{aligned}
R_{ijkl}\nabla_jf\nabla_lf
&=(\nabla_jf)\nabla_jR_{ik}-(\nabla_jf)\nabla_iR_{jk}\\
&=(\nabla_jf)\nabla_jR_{ik}+R_{jk}\nabla_i\nabla_jf\\
&=(\nabla_jf)\nabla_jR_{ik}+\lambda R_{ik}-R_{ij}R_{jk}\\
&=0.
\end{aligned}
\]
The last equality follows from \eqref{eq:soliton-flow-zero} and
\eqref{eq:soliton-Ric-polynomial}. Hence every sectional
curvature of a plane containing $\nabla f$ vanishes wherever
$\nabla f\neq0$.

The soliton is complete, has constant scalar curvature, and is
radially flat. Petersen--Wylie's rigidity theorem
\cite[Theorem~1.2]{PW}, which applies to both shrinking and expanding
solitons, therefore implies that its universal cover is an isometric
product
\[
N^{n-k}\times\mathbb R^k,
\qquad \Ric_N=\lambda g_N.
\]
After translating the Euclidean coordinates and adding a constant
to the lifted potential, we may normalize it as
\[
f=\frac{\lambda}{2}|x|^2.
\]
Since $f$ is nonconstant, the Euclidean factor has dimension
$k\geq1$. In particular, the product has parallel Ricci tensor, and
this property descends to $M$.

Let $\Gamma$ be the deck transformation group. The product
splitting and the lifted potential are preserved by $\Gamma$, whose
action on the Euclidean factor is orthogonal, as in the rigid-bundle
description of~\cite{PW}. We distinguish the two signs of $\lambda$.

If $\lambda>0$, a positive-dimensional factor $N$ is compact by the
Bonnet--Myers theorem. Every deck transformation preserves the
compact minimum set $N\times\{0\}$, so proper discontinuity implies
that $\Gamma$ is finite.

If $\lambda<0$, the factor $N$ is negative Einstein and may be
noncompact. The deck group can be infinite. For example, the product
of a closed hyperbolic surface of curvature $-1$ with
$\mathbb R^{n-2}$, equipped with $f=-|x|^2/2$, is an expanding
soliton with infinite deck group.

In either case, if $N$ is a point, every deck transformation fixes
the unique critical point of the potential, so the free action is
trivial. Conversely, every quotient in the non-steady cases has
parallel Ricci tensor and satisfies the soliton equation with the
displayed potential. Together with the steady case already treated,
this proves the theorem.
\end{proof}

\subsection{Quasi-Einstein manifolds}

\begin{proof}[Proof of Theorem~\ref{thm:quasi}]
Multiplication of $u$ by a positive constant does not change
\eqref{eq:quasi}, so we normalize $\max_Mu=1$.
By Lemma~\ref{lemma:cyclic}, the scalar curvature $R$ is constant.
Consequently, $\rho$ and $\tr(P)$ are also constant, with
\begin{equation}\label{eq:quasi-trace-P}
\tr(P)=(n-1)\lambda-(m+n-1)\rho.
\end{equation}
The standard quasi-Einstein identity
\[
P_{ij}\nabla_j u=\frac{u}{2}\nabla_i\rho
\]
therefore gives
\begin{equation}\label{eq:quasi-P-gradient-zero}
P_{ij}\nabla_j u=0.
\end{equation}

Since $\rho$ is constant, the cyclic condition can be written as
\begin{equation}\label{eq:quasi-P-cyclic}
\nabla_iP_{jk}+\nabla_jP_{ki}+\nabla_kP_{ij}=0.
\end{equation}
Differentiating \eqref{eq:quasi-P-gradient-zero}, we find
\begin{align}
(\nabla_iP_{jk})\nabla_k u
&=-P_{jk}\nabla_i\nabla_k u\notag\\
&=-\frac{u}{m}
  P_{jk}\bigl(P_{ik}+(\rho-\lambda)g_{ik}\bigr)\notag\\
&=-\frac{u}{m}
  \bigl(P_{ik}P_{kj}-(\lambda-\rho)P_{ij}\bigr).
\label{eq:quasi-differentiated-gradient}
\end{align}
Contracting \eqref{eq:quasi-P-cyclic} with $\nabla_k u$ and
applying \eqref{eq:quasi-differentiated-gradient} to its first
two terms gives the following pointwise identity:
\begin{equation}\label{eq:quasi-first-order}
\nabla_k u\,\nabla_kP_{ij}
=\frac{2u}{m}
 \bigl(P_{ik}P_{kj}-(\lambda-\rho)P_{ij}\bigr).
\end{equation}
In particular, the sign on the right-hand side follows from
$\nabla^2u=\frac{u}{m}(\Ric-\lambda g)$.

Taking the trace of \eqref{eq:quasi-first-order}, we obtain
\[
0=\langle\nabla u,\nabla\tr(P)\rangle
=\frac{2u}{m}
 \bigl(|P|^2-(\lambda-\rho)\tr(P)\bigr).
\]
Since $u>0$ in the interior of $M$, it follows that
\begin{equation}\label{eq:quasi-P-norm}
|P|^2=(\lambda-\rho)\tr(P).
\end{equation}
Hence $|P|^2$ is constant. Contracting
\eqref{eq:quasi-first-order} with $2P_{ij}$ now gives
\[
0=\langle\nabla u,\nabla|P|^2\rangle
=\frac{4u}{m}
 \bigl(\tr(P^3)-(\lambda-\rho)|P|^2\bigr),
\]
and therefore
\begin{equation}\label{eq:quasi-P-cube}
\tr(P^3)=(\lambda-\rho)|P|^2.
\end{equation}
The scalar $\tr(P^3)$ is consequently constant. A further
contraction of \eqref{eq:quasi-first-order}, this time with
$3P_{jk}P_{ki}$, yields
\[
0=\langle\nabla u,\nabla \tr(P^3)\rangle
=\frac{6u}{m}
 \bigl(\tr(P^4)-(\lambda-\rho)\tr(P^3)\bigr).
\]
Combining these identities, we conclude that
\begin{align*}
\left|P_{ik}P_{kj}-(\lambda-\rho)P_{ij}\right|^2
&=\tr(P^4)-2(\lambda-\rho)\tr(P^3)+(\lambda-\rho)^2|P|^2\\
&=0.
\end{align*}
Thus
\begin{equation}\label{eq:quasi-P-polynomial}
P_{ik}P_{kj}=(\lambda-\rho)P_{ij}.
\end{equation}
These identities, initially obtained in the interior, extend
to the boundary by continuity.

Substituting \eqref{eq:quasi-P-polynomial} into
\eqref{eq:quasi-first-order} and
\eqref{eq:quasi-differentiated-gradient}, respectively, gives
\begin{equation}\label{eq:quasi-two-zero-derivatives}
\nabla_k u\,\nabla_kP_{ij}=0,
\qquad
(\nabla_iP_{jk})\nabla_k u=0.
\end{equation}
We make the required specialization of \eqref{eq:Q-radial}
explicit. Commuting third derivatives in \eqref{eq:quasi},
and using the constancy of $\rho$, gives
\[
\begin{aligned}
R_{ijkl}\nabla_lu
={}&\frac{u}{m}(\nabla_iP_{jk}-\nabla_jP_{ik})\\
&+\frac1m\bigl((\nabla_i u)P_{jk}-(\nabla_j u)P_{ik}\bigr)\\
&+\frac{\rho-\lambda}{m}
  \bigl(g_{jk}\nabla_i u-g_{ik}\nabla_j u\bigr).
\end{aligned}
\]
Adding the remaining terms in \eqref{eq:Q-definition}, the
displayed algebraic terms cancel. The two terms left from
$P\owedge g$ vanish by \eqref{eq:quasi-P-gradient-zero}.
Consequently,
\[
Q_{ijkl}\nabla_l u
=\frac{u}{m}
 \bigl(\nabla_iP_{jk}-\nabla_jP_{ik}\bigr).
\]
Contracting this equation with $\nabla_j u$ and applying
\eqref{eq:quasi-two-zero-derivatives}, we obtain
\begin{equation}\label{eq:quasi-radial-Q-zero}
Q_{ijkl}\nabla_j u\nabla_l u=0.
\end{equation}
This is the radial $Q$-flatness condition. No division by
$|\nabla u|$ has been used, so it also holds at the critical
points of $u$.

We next determine the signs of the constants and the
multiplicities of the Ricci eigenvalues. If $\lambda-\rho=0$,
then \eqref{eq:quasi-P-polynomial} implies $P=0$, and
\eqref{eq:quasi} gives $\nabla^2u=0$. The boundary condition
and the maximum principle would then imply $u=0$, a
contradiction. Thus $\lambda-\rho\ne0$. Equation
\eqref{eq:quasi-P-polynomial} shows that the eigenvalues of
$P$ are $0$ and $\lambda-\rho$, with constant multiplicities.
Since $u$ is nonconstant, \eqref{eq:quasi-P-gradient-zero}
shows that the zero eigenvalue occurs. Put
\[
k=\dim\ker P,\qquad 1\leq k\leq n.
\]
At an interior maximum of $u$, its Hessian is nonpositive.
Its restriction to the nonzero subspace $\ker P$ is
$\frac{u}{m}(\rho-\lambda)g$. It follows that
$\lambda-\rho>0$. By \eqref{eq:quasi-trace-P},
\[
(n-k)(\lambda-\rho)
=\tr(P)
=(n-1)\lambda-(m+n-1)\rho.
\]
Consequently,
\begin{equation}\label{eq:quasi-rho-scaling}
\rho=\frac{k-1}{m+k-1}\lambda,
\qquad
\lambda-\rho=\frac{m}{m+k-1}\lambda.
\end{equation}
In particular, $\lambda>0$.

We now pass to the universal cover, equipped with the lifted
metric and potential. It is complete as a manifold with
boundary and simply connected; the lifted potential is
positive in its interior and vanishes on its boundary.
The structure is non-trivial, its scalar curvature is
constant, and $\lambda>0$. Moreover, the lifted curvature
tensor is bounded, because $M$ is compact, so it satisfies
the curvature-growth hypothesis in the rigidity theorem of
He--Petersen--Wylie. Equation
\eqref{eq:quasi-radial-Q-zero} is precisely their remaining
radial curvature hypothesis. Thus all the assumptions of
\cite[Theorem~1.11]{HPW14} are satisfied, and the universal
cover is rigid.

The classification of the rigid examples in
\cite[Proposition~2.6 and Table~2.1]{HPW14}, together with
\eqref{eq:quasi-rho-scaling}, gives
\[
\mathbb{S}^k_+\!
 \left(\sqrt{\frac{m+k-1}{\lambda}}\right)
 \times N^{n-k},
\qquad
\Ric_N=\lambda g_N.
\]
The factor $N$ is complete and simply connected. The lifted
potential depends only on the hemisphere factor and, under
our normalization, is its normalized height function:
\[
u=\cos\left(\sqrt{\frac{\lambda}{m+k-1}}\,r\right),
\]
where $r$ is the distance from the pole. In particular, the
Ricci tensor of the product, and hence that of $M$, is
parallel.

If $n-k=1$, the equation $\Ric_N=\lambda g_N$ is impossible
because $\lambda>0$ and a one-dimensional manifold has zero
Ricci curvature. If $n-k\geq2$, the Bonnet--Myers theorem
implies that $N$ is compact. The universal cover is therefore
compact in every case, so the covering group is finite. Its
isometric action is free and preserves the lifted potential.
When $k=n$, every such isometry of the hemisphere fixes the
unique maximum of $u$, so the free action must be trivial.
Thus this case gives the round hemisphere itself.

Finally, the case $k=1$ is included: the hemisphere factor is
the interval
\[
\left[-\frac{\pi}{2}\sqrt{\frac{m}{\lambda}},
       \frac{\pi}{2}\sqrt{\frac{m}{\lambda}}\right],
\]
with its standard metric and
$u=\cos(\sqrt{\lambda/m}\,r)$, where $r$ now denotes the
signed interval coordinate. This proves the theorem.
\end{proof}

\subsection{Vacuum static spaces and related critical spaces}

We now consider the unified critical equation \eqref{eq:critical}.
By Lemma~\ref{lemma:cyclic}, the condition $D=0$ implies that $R$ is
constant and $\nabla_j\mathring R_{ij}=0$. Taking the trace of
\eqref{eq:critical} gives
\begin{equation}\label{eq:critical-trace}
\Delta f=-\frac{R}{n-1}f+nb.
\end{equation}
The next proposition is obtained by computing the Laplacian of the
Hessian in two ways.

\begin{proposition}\label{prop:critical-tensor}
A solution of \eqref{eq:critical} with cyclic parallel Ricci tensor
satisfies the pointwise tensor identity
\begin{equation}\label{eq:critical-tensor-identity}
2(f+a)\Delta R_{ij}
+4\nabla_kf\,\nabla_kR_{ij}
+\left(\frac{aR}{n-1}+nb\right)\mathring R_{ij}=0.
\end{equation}
Consequently,
\begin{equation}\label{eq:critical-scalar-identity}
\begin{aligned}
\diver\left((f+a)^2\nabla|\mathring{\Ric}|^2\right)
={}&2(f+a)^2|\nabla\mathring{\Ric}|^2\\
&-\left(\frac{aR}{n-1}+nb\right)
 (f+a)|\mathring{\Ric}|^2.
\end{aligned}
\end{equation}
Both identities hold also on the level set $\{f+a=0\}$.
\end{proposition}

\begin{proof}
Since $R$ is constant, applying $\Delta$ directly to
\eqref{eq:critical} and using \eqref{eq:critical-trace}, we obtain
\begin{equation}\label{eq:critical-first-laplacian}
\begin{aligned}
\Delta(\nabla_i\nabla_jf)
={}&(f+a)\Delta R_{ij}
+2\nabla_kf\,\nabla_kR_{ij}\\
&+\left(-\frac{R}{n-1}f+nb\right)\mathring R_{ij}
-\frac{R}{n(n-1)}\Delta f\,g_{ij}.
\end{aligned}
\end{equation}
Here $\Delta\mathring R_{ij}=\Delta R_{ij}$ and
$\nabla_k\mathring R_{ij}=\nabla_kR_{ij}$, because $\nabla R=0$.

Next, we apply the Hessian--Laplacian formula
\eqref{eq:hessian-laplacian}. The traced equation gives
\[
\nabla_i\nabla_j\Delta f
=-\frac{R}{n-1}\nabla_i\nabla_jf.
\]
The cyclic condition is used in its uncontracted form to simplify
the first derivatives of the Ricci tensor:
\begin{equation}\label{eq:critical-derivative-block}
\nabla_iR_{jk}+\nabla_jR_{ik}-\nabla_kR_{ij}
=-2\nabla_kR_{ij}.
\end{equation}
For the curvature terms in \eqref{eq:hessian-laplacian},
substitution of \eqref{eq:critical} yields
\begin{equation}\label{eq:critical-curvature-block}
\begin{aligned}
&R_{ik}\nabla_k\nabla_jf+R_{jk}\nabla_i\nabla_kf
 -2R_{ikjl}\nabla_k\nabla_lf\\
&\qquad=2(f+a)\left(R_{ik}\mathring R_{kj}
 -R_{ikjl}\mathring R_{kl}\right)\\
&\qquad=2(f+a)\left(R_{ik}R_{kj}-R_{ikjl}R_{kl}\right)\\
&\qquad=-(f+a)\Delta R_{ij}.
\end{aligned}
\end{equation}
To verify the first equality, the terms proportional to the metric
in \eqref{eq:critical} cancel because
\[
R_{ik}g_{kj}+R_{jk}g_{ik}-2R_{ikjl}g_{kl}=0.
\]
The scalar parts of $\mathring R_{kj}$ and $\mathring R_{kl}$
cancel in the second equality. The last equality is exactly
\eqref{eq:cyclic-laplacian}, which was obtained by taking the
divergence of $D=0$.

Combining these three observations, the Hessian--Laplacian formula
becomes
\begin{equation}\label{eq:critical-second-laplacian}
\begin{aligned}
\Delta(\nabla_i\nabla_jf)
={}&-\frac{R}{n-1}\nabla_i\nabla_jf
 -(f+a)\Delta R_{ij}-2\nabla_kf\,\nabla_kR_{ij}\\
={}&-\frac{R}{n-1}(f+a)\mathring R_{ij}
 -(f+a)\Delta R_{ij}-2\nabla_kf\,\nabla_kR_{ij}\\
&-\frac{R}{n-1}
 \left(b-\frac{R}{n(n-1)}f\right)g_{ij}.
\end{aligned}
\end{equation}
The terms proportional to $g_{ij}$ in
\eqref{eq:critical-first-laplacian} and
\eqref{eq:critical-second-laplacian} agree, since
\[
-\frac{R}{n(n-1)}\Delta f
=-\frac{R}{n-1}\left(b-\frac{R}{n(n-1)}f\right).
\]
Comparison of the remaining terms proves
\eqref{eq:critical-tensor-identity}.

We next contract \eqref{eq:critical-tensor-identity} with
$\mathring R_{ij}$. Using
\[
\mathring R_{ij}\Delta R_{ij}
=\frac12\Delta|\mathring{\Ric}|^2-|\nabla\mathring{\Ric}|^2,
\qquad
\mathring R_{ij}\nabla_kR_{ij}
=\frac12\nabla_k|\mathring{\Ric}|^2,
\]
we obtain
\begin{equation}\label{eq:critical-scalar-before-weight}
\begin{aligned}
&(f+a)\Delta|\mathring{\Ric}|^2
 +2\langle\nabla f,\nabla|\mathring{\Ric}|^2\rangle\\
&\qquad=2(f+a)|\nabla\mathring{\Ric}|^2
 -\left(\frac{aR}{n-1}+nb\right)|\mathring{\Ric}|^2.
\end{aligned}
\end{equation}
Multiplication by $f+a$ and the product rule give
\eqref{eq:critical-scalar-identity}. No division by $f+a$ occurs
in this argument, so the formulas hold on its zero set as well.
\end{proof}

\begin{remark}
Because $R$ is constant,
$|\Ric|^2=|\mathring{\Ric}|^2+R^2/n$ and
$|\nabla\Ric|^2=|\nabla\mathring{\Ric}|^2$. Thus
\eqref{eq:critical-scalar-identity} can equivalently be written with
$\nabla|\Ric|^2$ on the left and $|\nabla\Ric|^2$ on the right.
The rigidity consequences below use this pointwise differential
identity together with integration over the compact manifold.
\end{remark}

\begin{remark}[The same comparison for solitons and quasi-Einstein manifolds]
\label{rem:hessian-comparison}
The comparison of two expressions for the Laplacian of the Hessian in
Proposition~\ref{prop:critical-tensor}, using the cyclic parallel Ricci
condition, applies equally to gradient Ricci solitons and quasi-Einstein
manifolds. To avoid repeating this second-order computation for each
class, the preceding proofs of Theorems~\ref{thm:soliton} and
\ref{thm:quasi} use first-order pointwise identities and trace
contractions instead. We record the resulting differential identities.

For gradient Ricci solitons and quasi-Einstein manifolds, respectively,
\begin{align}
\diver\!\left(e^{-f}\nabla|\Ric|^2\right)
&=2e^{-f}|\nabla\Ric|^2,
\label{eq:remark-soliton-divergence}\\
\diver\!\left(u^{m+1}\nabla|\Ric|^2\right)
&=2u^{m+1}|\nabla\Ric|^2.
\label{eq:remark-quasi-divergence}
\end{align}
The second identity holds in $\operatorname{int}(M)$, where $u>0$.
In both cases, the preceding trace identities imply that $|\Ric|^2$
is constant. Thus the left-hand sides vanish, and positivity of the
weights gives $\nabla\Ric=0$ pointwise. In the quasi-Einstein case,
the conclusion extends to the boundary by continuity.
\end{remark}

\subsubsection{Closed vacuum static spaces}

\begin{proof}[Proof of Theorem~\ref{thm:static}]
For a vacuum static space, the constants in the critical equation are
$a=b=0$. Thus
\begin{equation}
\nabla_i\nabla_j f
=f\left(R_{ij}-\frac{R}{n-1}g_{ij}\right),
\qquad
\Delta f=-\frac{R}{n-1}f.
\end{equation}
Since $D=0$, the scalar curvature is constant. Multiplying the second
equation by $f$ and integrating over the closed manifold gives
\begin{equation}
\int_M |\nabla f|^2\,dV
=\frac{R}{n-1}\int_M f^2\,dV.
\end{equation}
The potential is nonconstant, so the left-hand side is positive.
Consequently, $R>0$.

Putting $a=b=0$ in the scalar identity of
Proposition~\ref{prop:critical-tensor} and integrating, we obtain
\begin{equation}
0=2\int_M f^2|\nabla\mathring{\Ric}|^2\,dV
 =2\int_M f^2|\nabla\Ric|^2\,dV.
\end{equation}
Hence $\nabla\Ric=0$ at every point where $f\ne0$.
We verify that this set is dense without imposing an analyticity
assumption. If $f$ vanished on a nonempty open set, choose a point in
that set. Both $f$ and $\nabla f$ vanish there. Along any unit-speed
geodesic issuing from this point, the static equation reads
\begin{equation}
\frac{d^2}{dt^2}f(\gamma(t))
=\left(\Ric(\dot\gamma(t),\dot\gamma(t))
       -\frac{R}{n-1}\right)f(\gamma(t)).
\end{equation}
Uniqueness for this linear ordinary differential equation, with zero
initial value and zero initial derivative, shows that $f$ vanishes
along the entire geodesic. Completeness and connectedness then imply
$f\equiv0$, a contradiction. Therefore $\{f\ne0\}$ is dense, and
continuity yields
\begin{equation}
\nabla\Ric=0\qquad\text{on }M.
\end{equation}

We now give the geometric classification directly. Since the Ricci
tensor is parallel, its eigenvalues are constant and its
eigendistributions are parallel. By the de Rham decomposition theorem,
the complete simply connected manifold splits as a Riemannian product
of the factors corresponding to these eigendistributions. At least
one eigenvalue differs from $R/(n-1)$; otherwise, taking the trace would
give $R=nR/(n-1)$, contrary to $R>0$. Choose such an eigenspace and
denote its constant eigenvalue by $R_{11}$ in an adapted local
orthonormal frame. In the following, $R_{11}$ always denotes this
chosen eigenvalue.

Let the chosen factor have dimension $k$, and combine the remaining
factors into $N^{n-k}$. Thus, at this stage, the product is written as
$M^k\times N^{n-k}$. We first show that $f$ depends only on $M^k$.
The mixed components of the static equation give
$\nabla_i\nabla_l f=0$ when $i$ is tangent to $M^k$ and $l$ is tangent
to $N$. If $i,j$ are tangent to $M^k$ and $l$ is tangent to $N$, the
product connection and the vanishing of mixed curvature give
\begin{equation}
\nabla_l\nabla_i\nabla_j f
=\nabla_i\nabla_l\nabla_j f=0.
\end{equation}
Differentiating the static equation in the same directions gives
\begin{equation}
0=\left(R_{11}-\frac{R}{n-1}\right)
  (\nabla_l f)g_{ij}.
\end{equation}
Our choice of eigenvalue implies $\nabla_l f=0$. Thus $f$ is independent
of $N$. The static equation restricted to $N$ consequently gives
\begin{equation}
0=f\left(R_{ij}-\frac{R}{n-1}g_{ij}\right)
\qquad (i,j\text{ tangent to }N).
\end{equation}
Evaluating at a point of $M^k$ where $f\ne0$, we obtain
\begin{equation}
\Ric_N=\frac{R}{n-1}g_N.
\end{equation}

On the chosen factor, the nonconstant function $f$ satisfies
\begin{equation}
\nabla_i\nabla_j f
=\left(R_{11}-\frac{R}{n-1}\right)fg_{ij}.
\end{equation}
Here the coefficient is strictly negative. To see this using only
completeness and boundedness, take a complete unit-speed geodesic
through a point where $f\ne0$. Along this geodesic,
\begin{equation}
\frac{d^2}{dt^2}f(\gamma(t))
=\left(R_{11}-\frac{R}{n-1}\right)f(\gamma(t)).
\end{equation}
A positive coefficient admits no bounded nonzero solution on the
whole real line. This contradicts the boundedness of $f$ on the closed
manifold. The coefficient is nonzero by construction, and therefore
\begin{equation}
R_{11}<\frac{R}{n-1}.
\end{equation}

The chosen factor cannot be one-dimensional: a complete simply
connected one-dimensional factor is $\mathbb R$, which would make the
product noncompact. Hence $k\ge2$. Applying Obata's theorem
\cite[Theorem~A]{Obata} to this complete factor shows that it is a round
sphere of sectional curvature $R/(n-1)-R_{11}$. Its Ricci curvature is
therefore
\begin{equation}
R_{11}g_{ij}
=(k-1)\left(\frac{R}{n-1}-R_{11}\right)g_{ij}.
\end{equation}
It follows that
\begin{equation}
R_{11}=\frac{(k-1)R}{k(n-1)},
\qquad
\frac{R}{n-1}-R_{11}=\frac{R}{k(n-1)}.
\end{equation}
Consequently,
\begin{equation}
(M^n,g)
\cong
\mathbb S^k\!\left(\sqrt{\frac{k(n-1)}{R}}\right)
\times N^{n-k},
\qquad
\Ric_N=\frac{R}{n-1}g_N.
\end{equation}
The factor $N$ is closed and simply connected, and it is understood to
be a point when $k=n$. A one-dimensional $N$ is impossible because its
Ricci tensor would vanish, whereas $R/(n-1)>0$. The potential depends
only on the spherical factor and satisfies the corresponding Obata
equation; in particular, it is a nonzero multiple of a height
function. This proves the theorem.
\end{proof}

\begin{remark}
The conclusion $\nabla\Ric=0$ does not require simple connectivity.
For an arbitrary closed connected vacuum static space with a
nonconstant potential and $D=0$, apply the product argument to its
universal cover. The lifted potential remains bounded, so all the
preceding arguments apply up to the exclusion of a one-dimensional
factor. If $k\ge2$, the universal cover is
\begin{equation}
\mathbb S^k\!\left(\sqrt{\frac{k(n-1)}{R}}\right)
\times N^{n-k},
\qquad
\Ric_N=\frac{R}{n-1}g_N.
\end{equation}
The factor $N$ is complete and simply connected, and its positive
Einstein constant makes it compact by Myers' theorem when it has
positive dimension. Thus the deck group in this case is finite.
If $k=1$, the universal cover is
\begin{equation}
\mathbb R\times N^{n-1},
\qquad
\Ric_N=\frac{R}{n-1}g_N,
\end{equation}
and, after translating the real coordinate, the lifted potential is
a nonzero constant multiple of
$\cos\bigl(\sqrt{R/(n-1)}\,t\bigr)$. The original manifold is obtained
by a group of deck transformations preserving this potential. Such
transformations may couple translations in the real factor with
isometries of $N$.

The nonconstancy assumption distinguishes these cases from a nonzero
constant static potential. For such a constant potential, the trace
of the static equation gives $R=0$, and the full equation then gives
$\Ric=0$.
\end{remark}

\subsubsection{CPE metrics}

\begin{proof}[Proof of Theorem~\ref{thm:cpe}]
Here $a=1$ and $b=0$. Although the theorem fixes
$R=n(n-1)>0$, positivity also follows from nonconstancy of $f$:
multiplication of \eqref{eq:critical-trace} by $f$ and integration
give
\[
\int_M|\nabla f|^2\,dV
=\frac{R}{n-1}\int_Mf^2\,dV>0.
\]

Since $\nabla_i\mathring R_{ij}=0$, the CPE equation implies
\begin{align*}
\nabla_i(\mathring R_{ij}\nabla_jf)
&=\mathring R_{ij}\nabla_i\nabla_jf\\
&=(f+1)|\mathring{\Ric}|^2.
\end{align*}
Integrating this identity over the closed manifold, we obtain
\begin{equation}\label{eq:cpe-linear-weight}
\int_M(f+1)|\mathring{\Ric}|^2\,dV=0.
\end{equation}
This is consistent with Lemma 2 in ~\cite{Zhang}.

By Proposition~\ref{prop:critical-tensor},
\[
\begin{aligned}
\diver\left((f+1)^2\nabla|\mathring{\Ric}|^2\right)
={}&2(f+1)^2|\nabla\mathring{\Ric}|^2\\
&-\frac{R}{n-1}(f+1)|\mathring{\Ric}|^2.
\end{aligned}
\]
Integrating and using \eqref{eq:cpe-linear-weight} yields
\[
0=2\int_M(f+1)^2|\nabla\mathring{\Ric}|^2\,dV.
\]
Consequently, $\nabla\Ric=0$ on $\{f+1\neq0\}$.

This set is dense in $M$. Indeed, if $f=-1$ on a nonempty open
set, then $\nabla^2f=0$ there, and \eqref{eq:critical} gives
\[
-\frac{R}{n(n-1)}g_{ij}=0,
\]
contrary to $R>0$. Continuity now implies $\nabla\Ric=0$
everywhere. Returning to \eqref{eq:critical-tensor-identity}, we
obtain
\[
\frac{R}{n-1}\mathring R_{ij}=0.
\]
Thus the metric is Einstein and the CPE equation reduces to
\[
\nabla_i\nabla_jf=-\frac{R}{n(n-1)}fg_{ij}.
\]
Obata's theorem~\cite{Obata}, applied to the nonconstant function
$f$, identifies $(M^n,g)$ with
$\mathbb S^n(\sqrt{n(n-1)/R})$. Under the stated normalization,
this is the unit round sphere $\mathbb S^n(1)$.
\end{proof}

\subsubsection{Miao--Tam critical metrics}

\begin{proof}[Proof of Theorem~\ref{thm:mt}]
Taking $a=0$ and $b=-1/(n-1)$ in
\eqref{eq:critical-scalar-identity}, we have
\begin{equation}\label{eq:mt-divergence}
\diver\left(f^2\nabla|\mathring{\Ric}|^2\right)
=2f^2|\nabla\mathring{\Ric}|^2
+\frac{n}{n-1}f|\mathring{\Ric}|^2.
\end{equation}
Since $f=0$ on the smooth boundary, the boundary integral in the
divergence theorem vanishes. Therefore,
\begin{equation}\label{eq:mt-final-integral}
0=2\int_Mf^2|\nabla\mathring{\Ric}|^2\,dV
+\frac{n}{n-1}\int_Mf|\mathring{\Ric}|^2\,dV.
\end{equation}
Both integrands are nonnegative. The assumption $f>0$ in the
interior implies $\mathring{\Ric}=0$ there, and then everywhere
by continuity. Hence $(M^n,g)$ is Einstein.

The classification theorem for Einstein Miao--Tam critical metrics
\cite[Theorem~1.1]{MT11} now shows that $(M^n,g)$ is isometric to
a geodesic ball in a simply connected space form. Its sectional
curvature is $R/[n(n-1)]$, so the ambient space is Euclidean,
hyperbolic, or spherical according to the sign of $R$.
No restriction on the sign of $R$ was used in
\eqref{eq:mt-final-integral}.
\end{proof}
\section{Acknowledgement}
The author would like to express his sincere gratitude to his advisor, Jian Ye.

\end{document}